[1994/12/01]
\documentclass[draft]{article}
\title{
Short-time stability of scalar viscous shocks in the inviscid limit by the relative entropy method}
\author{
Kyudong Choi\thanks{kchoi@math.wisc.edu, University of Wisconsin-Madison},
Alexis F. Vasseur\thanks{vasseur@math.utexas.edu, University of Texas at Austin}
}
\date{\today }

\usepackage{amsmath,amsthm,amsfonts,amsopn,verbatim}

\chardef\bslash=`\\ 

\newtheorem{thm}{Theorem}[section]

\newtheorem{lem}[thm]{Lemma}
\newtheorem{prop}[thm]{Proposition}

\newtheorem*{thm*}{Theorem}
\newtheorem*{lem*}{Lemma}
\newtheorem*{prop*}{Proposition}
\theoremstyle{definition}

\newtheorem*{ack}{Acknowledgment:}

\theoremstyle{remark}
\newtheorem{rem}{Remark}[section]

\newcommand{\R}{\mathbb{R}}
\newcommand{\eps}{\varepsilon}
\newcommand{\dis}{\displaystyle}
\newcommand{\C}{C^*}

\newcommand{\He}{H}

\usepackage{amscd,amsthm}
\usepackage{}

\def\cprime{$'$}

\newcommand{\eval}[2][\right]{\relax
  \ifx#1\right\relax \left.\fi#2#1\rvert}

\begin{document}

\maketitle
\renewcommand{\sectionmark}[1]{}
\begin{abstract}
We consider  inviscid limits to  shocks  
for   viscous scalar conservation laws 
in one space dimension, with  strict convex fluxes. We show that we 
can obtain sharp estimates in $L^2$
for a class of large perturbations
and for any bounded time interval. 
Those perturbations can be chosen big enough to destroy the viscous layer. 
This shows that the fast convergence to the shock does not depend on the 
fine structure of the viscous layers. This is the first application of the relative
 entropy method developed in 
\cite{nick}, \cite{nick_vasseur} to the study of an  
inviscid
limit to a shock. \\

\end{abstract}
\noindent Keywords: viscous scalar conservation laws; inviscid limits; stability; relative
entropy method; shocks.\\

\noindent AMS Subject Classification: 35B40, 35L65, 35L67.\\

\section{Introduction and the main result}\label{sec_introduction}

For any strictly convex flux function $A\in C^2(\R)$, we consider the family of 
 viscous scalar conservation laws in one space dimension:
\begin{equation}\label{scalar_cl}\begin{cases}
&\partial_tU+\partial_xA(U)=\eps\partial_{xx}^2U\quad \mbox{ for } t>0,x\in\mathbb{R},\\
 &U(0,x)= U_0(x)  \mbox{ for } x\in\mathbb{R},
\end{cases}
\end{equation}   
for any $\eps>0$ and  $U_0\in L^\infty$. 
Global unique solutions to (\ref{scalar_cl}) have been constructed by 
Hopf \cite{MR0047234}  and 
 Ole{\u\i}nik \cite{MR0094541}. The inviscid case, $\eps=0$, is covered by the theory of 
Kru{\v{z}}kov \cite{Kruzkov}. Kuznetsov showed in \cite{Kuznetsov} that, for fixed initial
 data $U_0$, the solutions $U^{\eps}$ of (\ref{scalar_cl})
 converge in $L^1$, when $\eps$ goes to zero,  to the solution $U^0$ of 
 the inviscid scalar conservation law (the equation  (\ref{scalar_cl}) with $\eps=0$) with the rate $\sqrt{\eps}$:
 $$\|U^\eps(t)-U^0(t)\|_{L^1}\leq C\sqrt{\eps t}\,\mbox{TV}\,({U_0})$$ (for the proof, e.g. see either Serre \cite{Serre}
 or Perthame \cite{Perthame}).\\
 

 In this paper we consider the 
inviscid 
 limit for general initial values and for any bounded time interval. 
We are particularly interested in the cases where the initial values carry 
too much entropy  for the structure of the layer to be preserved 
in its vanishing viscosity limit.
The shocks solutions of the inviscid case ($\eps=0$) can be described as follows.  
Consider two constants $C_L>C_R$, and the associated function defined by 
 \begin{equation}\label{def_S}
 S_0(x)=\begin{cases}
&C_L \mbox{ if } x<0,\\
&C_R \mbox{ if } x\geq 0.
\end{cases}
\end{equation}
Then, the Rankine-Hugoniot conditions ensures that the function
\begin{equation}\label{def_sigma}
S_0(x-\sigma t) \qquad \mbox{with }\quad\sigma:=\frac{A(C_L)-A(C_R)}{C_L-C_R},
\end{equation}
is a solution to the inviscid equation (\ref{scalar_cl}) with $\eps=0$. 
The condition $C_L>C_R$ implies that they verify the entropy conditions, that is:
$$
\partial_t \eta(U)+\partial_x G(U)\leq 0,\qquad t>0, \ x\in \R,
$$
for any convex functions $\eta$, and 
$G'=\eta'A'.$\\

Our main result is the following.
\begin{thm}\label{main_thmL2} 
Let $C_L>C_R$  and $U_0\in L^\infty(\R)\cap BV_{loc}(\mathbb{R}) $ be such that 
$$
(U_0-S_0)\in L^2(\R)\qquad\mathrm{ and}\qquad  
  (\frac{d}{dx}U_0)_+\in L^2(\R).
 $$ 
Then, there exists $\eps_0>0$ such that for any $T>0$, we have a constant $\C>0$ with the following: 
 \begin{itemize}
 \item[I.]
For any  $U$ solution to  
 \eqref{scalar_cl} with $0<\eps\leq \eps_0$, 
there exists a curve $X\in L^\infty(0,T)$ such that $X(0)=0$ and  for any $0<t<T$:
\begin{equation}\begin{split}\label{main_estiamteL2}
 &\|U(t)-S(t)\|^2_{L^2(\R)}\leq
\|U_0-S_0\|^2_{L^2(\R)}
 +\C  \eps\log({1}/{\eps}),
\end{split}\end{equation}
 where $S(t,x):=S_0(x-X(t))$, and $S_0$ is defined by (\ref{def_S}).

\item[II.] Moreover, this curve satisfies
\begin{equation}\label{estimate_curve_lishitzL2}
|\dot{X}(t)|\leq \C \quad\mbox{ and}
\end{equation}
\begin{equation}\label{estimate_curve_w.r.t._L^2_differL2}
|X(t)-\sigma t|^2\leq \C t^{2/3}
\left(\|U_0-S_0\|^2_{L^2(\mathbb{R})}
+\eps\log({1}/{\eps})\right).
 \end{equation}
 
\item[III.] The constant $\eps_0$ depends only on 
$\|(\frac{d}{dx}U_0)_+\|_{L^2}$, $C_L$, $C_R$, $\|U_0\|_{L^\infty}$ and
the flux function $A$, while $\C$ depends only on the same set as well as T.
  \end{itemize}
 \end{thm}

\begin{rem} 
 For any continuous function $g$, we define the function 
$g_+$ by
$g_+(x):=g(x)\cdot\chi_{\{g>0\}}(x)$ where $\chi_{\{g>0\}}$ is the characteristic function on the positive part of the function $g$. In our theorem, the assumption $U_0\in BV_{loc}$ 
ensures that  $\frac{d}{dx}U_0$ is a Radon
measure. Hence,  $(\frac{d}{dx}U_0)_+$ is also a Radon measure, and the 
condition $(\frac{d}{dx}U_0)_+\in L^2$ makes sense. Note that our estimates do not depend on any local  $BV$ norms of 
$U_0$.
\end{rem}
\begin{rem}\label{L_P_rem} 
 The condtion $ (\frac{d}{dx}U_0)_+\in L^2(\R)$ can be replaced with $ (\frac{d}{dx}U_0)_+\in L^p(\R)$
 for any $1<p\leq\infty$. Indeed, as in Lemma \ref{lem_decreasing_L_2_positive_derivative}, it can be shown that
 $\|(\partial_xU(t))_+\|_{L^p(\mathbb{R})}$ is non-increasing in time (see Remark \ref{L_P_lemma}).
 The only place where the assumption $ (\frac{d}{dx}U_0)_+\in L^2(\R)$ is used is in the estimate 
 \eqref{L_2_cond_used} in the proof
 of Proposition \ref{prop_Hyp}. In order to use $ (\frac{d}{dx}U_0)_+\in L^p(\R)$
 for any $1<p\leq\infty$, one needs to have $(\eps\delta)^{1-1/p}$ instead of $\sqrt{\eps\delta}$ in \eqref{L_2_cond_used}.
\end{rem}
\begin{rem} 
The term $\sigma t$ in the estimate \eqref{estimate_curve_w.r.t._L^2_differL2} is meaningful when
$t\gg (\eps \log(1/\eps))
^3$.
\end{rem}
This result shows a rate of convergence  slightly worse than $\eps$ (to the
$\log$), for the inviscid limit to a shock, measured via  the $L^2$ norm (squared).  
In the case of the limit to a regular solution of the inviscid case, 
the rate of convergence is $\sqrt{\eps}$ (see \cite{vass:handbook}, for instance).
We also refer to 
Goodman and Xin \cite{MR1188982},
Bressan, Liu and Yang \cite{MR1723032},
Lewicka \cite{MR2091511},
 Bressan and Yang \cite{MR2053759}, 
Christoforou and Trivisa \cite{MR2861664}.

\vskip0.3cm
An easy layer study shows that $\eps$ is the optimal rate for shocks with  special 
initial data. Indeed, one can construct an associated steady viscous
 layer (see for example Il{\cprime}in and Ole{\u\i}nik \cite{Oleinik}) $S_1$
solution to 
\begin{equation}\label{layer}
\begin{cases}
&\dis{A(S_1)-A(C_L)-\sigma(S_1-C_L)=S_1',\qquad x\in\R,}\\
&\dis{\lim_{x\to-\infty}S_1=C_L,\qquad \lim_{x\to+\infty}S_1=C_R.}
\end{cases}
\end{equation}
It is easy to show that $S_1((x-\sigma t)/\eps)$ is a solution to (\ref{scalar_cl})
 with initial data 
$S_1(x/\eps)$. In this case, the rate of convergence is of order $\eps$ since:
$$
\int_\R|S_1((x-\sigma t)/\eps)-S_0(x-\sigma t)|^2\,dx
=
 \eps \int_{\R}|S_1(x)-S_0(x)|^2\,dx=C\eps.
$$
This layer study can be extended to the case of small initial perturbation where:
$$
\int_\R|U_0(x)-S_0(x)|^p\,dx\leq C\eps,
$$
for a $1\leq p<\infty$.
In this case, for a  solution $U$ to \eqref{scalar_cl}, we can consider 
$$
V(t,x)=U(\eps t,\eps x),
$$
and study the asymptotic for large time.
The function $V$ is a solution to the equation 
\begin{equation*}
\begin{cases}&\partial_t V+\partial_x A(V) -\partial^2_{xx}V=0,\\
&V(0,x)=U(0,\eps x). \end{cases}
\end{equation*}
The convergence to $S_1$, up to a (constant) drift,
 in this setting,  has been extensively studied 
(see for instance  \cite{Oleinik}, Freist{\"u}hler and Serre \cite{Freis_Serre},
 Kenig and Merle 
\cite{Kenig_Merle}).
In this situation of small perturbation of the initial shock, those results show 
that the convergence with rate 
$\eps$ for the system (\ref{scalar_cl}) is due
to the asymptotic limit in large time  of the layer 
function $U(\cdot/\eps)$ to $S_1((\cdot-\sigma t)/\eps)$.

 \vskip0.3cm
 
 This layer study, however, collapses when
 $$
 \int_\R |U_0(x)-S_0(x)|^2\,dx\gg \eps.
 $$
In this situation, there is too much entropy for the asymptotic limit 
of the layer structure to be true. The physical layer may be destroyed. 
Theorem \ref{main_thmL2} shows that, 
nevertheless, the sharp convergence (up to the $\log$) still holds for any bounded time interval. 
\vskip0.3cm
Taking a limit as $\eps$ goes to $0$ in Theorem  \ref{main_thmL2}, we recover
 the $L^2$ stability  of 
shocks (up to a drift) first showed by Leger in \cite{nick}.  Note that the stability 
result has to be up to a drift which depends on the solution itself (and may be not unique). 
This feature is also true for our result. 
The drift cannot be taken constant, as in the case of the layer problem. 

\vskip0.3cm

Our result  is based on the relative entropy method first used by Dafermos and DiPerna
 to show $L^2$ stability and uniqueness of Lipschitzian solutions to conservation laws
\cite{Dafermos4,Dafermos1,DiPerna}. They showed, in particular, that if $\overline{U}$
 is a Lipschitzian solution of a suitable conservation law on a lapse of time $[0,T]$, 
then for any bounded weak entropic solution $U$ it holds:
\begin{equation}\label{entropy_decrease}
\int_{\R}|U(t)-\overline{U}(t)|^2\,dx  \leq C \int_{\R}|U(0)-\overline{U}(0)|^2\,dx,
\end{equation}
for a constant $C$ depending on $\overline{U}$ and $T$.\\

The relative entropy method is also an  important tool in the study of asymptotic
 limits ($\eps \rightarrow 0$). The main idea is that convergence holds thanks to the strong stability of the 
solutions of the limit equations. Roughly speaking, if we have good consistency of $\eps$
 models, with respect to the limit one, then non linearities are driven by the strong
 stability of the solution of the limit equation. 
Applications of the relative entropy method in this context began with the work of Yau
 \cite{Yau} and have been studied by many others. For incompressible limits, see Bardos,
 Golse, Levermore \cite{Bardos_Levermore_Golse1,Bardos_Levermore_Golse2}, Lions and 
 Masmoudi \cite{Lions_Masmoudi}, Saint Raymond et al. 
\cite{SaintRaymond1,SaintRaymond2,SaintRaymond3,SaintRaymond4}. For  compressible models, 
see Tzavaras \cite{Tzavaras_theory} in the context of relaxation 
and \cite{BV,BTV,MV} in the context of hydrodynamical limits. However, in all 
those cases, the method works as long as the limit solution is Lipschitz.  This is 
due to the fact that strong stability as 
(\ref{entropy_decrease}) is not true when $\overline{U}$ has a discontinuity. It has 
been proven in  \cite{nick, nick_vasseur}, however, that some shocks are  strongly 
stable up to a shift (see also related works from Chen and Frid  \cite{Frid2,Frid1} 
and Chen, Frid 
and Li \cite{Chen1}).  This article is the first extension  of 
those results of stability, to the study of  
inviscid
limits to a shock. 
This is a part of the program initiated in \cite{vass:handbook}.
\vskip0.3cm
The result can be extended to any entropy in the following way. 
Fix any strictly convex function $\eta\in C^2$ as an entropy. 
We define the associated  relative entropy functional $\eta(\cdot|\cdot)$ as
$$
\eta(x|y):=\eta(x)-\eta(y)-\eta^\prime(y)(x-y).
$$
We then have the following extension.
\begin{thm}\label{main_thm} Consider a strictly convex entropy functional $\eta\in C^2(\R)$. 
Let $C_L>C_R$  and $U_0\in L^\infty(\R)\cap BV_{loc}(\mathbb{R}) $ be such that 
$$
(U_0-S_0)\in L^2(\R)\qquad\mathrm{ and}\qquad  
  (\frac{d}{dx}U_0)_+\in L^2(\R).
 $$ 
 Then, there exists $\eps_0>0$ such that for any $T>0$, we have a constant $\C>0$ with the following: 
 \begin{itemize}
 \item[I.]
For any  $U$ solution to  
 \eqref{scalar_cl} with $0<\eps\leq \eps_0$,  
there exists a curve $X\in L^\infty(0,T)$ such that $X(0)=0$, and  for any $0<t<T$, 
and for any $\alpha$ verifying $\eps\leq \alpha\leq\eps_0$, we have:
\begin{equation}\begin{split}\label{main_estiamte}
 &\int_{\{|x-X(t)|\geq \C\alpha\}}\eta(U(t,x)|S(t,x))\,dx\leq
\int_\R\eta(U_0(x)|S_0(x))\,dx
 +\C e^{-\alpha/\eps},
\end{split}\end{equation}
 where $S(t,x):=S_0(x-X(t))$, and $S_0$ is defined by (\ref{def_S}).

\item[II.] Moreover, this curve satisfies
\begin{equation}\label{estimate_curve_lishitz}
|\dot{X}(t)|\leq \C \quad\mbox{ and}
\end{equation}
\begin{equation}\label{estimate_curve_w.r.t._L^2_differ}
|X(t)-\sigma t|^2\leq \C t^{2/3}
\left(\int_\R \eta(U_0(x)|S_0(x))\,dx
+\eps\log(1/\eps)\right).
 \end{equation}

  \item[III.] The constant $\eps_0$ depends only on 
$\|(\frac{d}{dx}U_0)_+\|_{L^2}$, $C_L$, $C_R$, $\|U_0\|_{L^\infty}$, 
the flux function $A$ and the entropy functional $\eta$, while $\C$ depends only on 
the same set as well as T.
  \end{itemize}
 \end{thm}
 \begin{rem}\label{L_P_rem2} 
 As in Remark \ref{L_P_rem}, the condtion $ (\frac{d}{dx}U_0)_+\in L^2(\R)$ can be replaced with $ (\frac{d}{dx}U_0)_+\in L^p(\R)$
 for any $1<p\leq\infty$.  
 \end{rem}
Theorem \ref{main_thmL2} is a direct application of Theorem \ref{main_thm} with 
$\eta(x):=x^2$, and $\alpha=\eps\log(1/\eps)$. Indeed, in this case we have 
$\eta(x|y)=(x-y)^2$, and 
$$
\int_{\{|x-X(t)|\leq \C\alpha\}}\eta(U(t,x)|S(t,x))\,dx\leq 
C |\{|x-X(t)|\leq \C\alpha\}|\leq C \C \alpha.
$$
\vskip0.3cm
For the rest of the paper, we will assume that the 
initial value $U_0$ lies not only $BV_{loc}$ but also $C^1$. It allows us to work with
smooth solutions $U\in C^1([0,T]\times \R)$.  The general 
 $BV_{loc}$ case can be obtained by a density argument.
\vskip0.3cm
The idea of the proof is to study the evolution of the relative entropy of the solution with respect to the shock, outside of a small region centered at $X(t)$ (this small region corresponds to the layer localization):
\begin{equation}\label{def_quantity}
\int_{-\infty}^{X(t)-\delta\eps}\eta(U(t,x)|C_L)\,dx+\int_{X(t)+\delta\eps}^{\infty}\eta(U(t,x)|C_R)\,dx.
\end{equation}
The change in time involves two effects. One is due to the hyperbolic part of the equation, and the second involves the parabolic part (or order $\eps$). 
In \cite{nick}, it was shown that,
for the hyperbolic case $\eps=0$, with zero layer width $\delta=0$, the quantity 
\eqref{def_quantity} 
is non-increasing when we choose wisely the drift $X(t)$. When considering  the viscous term, the layer with width $(\delta\eps)$ is introduced  
to avoid the effect of the  viscous term on the layer (see Lemma \ref{baby_lemma}). 
 The idea is then that the stability induced by the hyperbolic part is enough to 
counterbalance the effect of the parabolic term,
  provided that we consider a layer fat enough
(see Proposition \ref{prop_Hyp} and the proof of Proposition \ref{prop_diffusion}). For technical considerations, we will fix $\delta=\log(1/\eps)$. The drift $X(t)$ is still chosen with respect to the hyperbolic part of the equation in a similar way
 as in \cite{nick}. Stability   is preserved, despite the non zero layer width, thanks to a monotonicity property 
induced in the layer  by the additional assumption $(\frac{d}{dx}U_0)_+\in L^2(\R)$.

\section{Evolution of the relative entropy}

For $\delta>0$, we consider a Lipschitz nondecreasing 
function $\phi$ to localize the layer, verifying
\begin{equation*}
\phi(x)=\begin{cases}&0\qquad \mathrm{if} \ \ x\leq 0,\\
&1\qquad \mathrm{if} \ \ x\geq \delta.
\end{cases}\end{equation*}
To get the optimal result, we will later fix a special function (see \eqref{pd}). \\

For any fixed $\delta>0$ and $X\in C^1([0,T])$, we are interesting in the evolution  of 
\begin{equation}\label{def_H}
\He(t):=\int_{-\infty}^{\infty}\phi^2(|x-X(t)|/\eps)\eta(U(t,x)|S(t,x))\,dx,
\end{equation}
where $S(t,x):=S_0(x-X(t))$ and where $S_0$ is defined in \eqref{def_S}. A special value of $\delta$ (depending on $\eps$),  and of the function $X$ will be chosen later.
Note that $\He(t)$ controls the quantity \eqref{def_quantity}. In fact, we have $\eqref{def_quantity}\leq\He(t)$.
\vskip0.3cm
Let us denote $F(\cdot,\cdot)$ the flux of the relative entropy $\eta(\cdot|\cdot)$  defined by
\begin{equation}\begin{split} 
  &F(x,y):=G(x)-G(y)-\eta^\prime(y)(A(x)-A(y)).
 \end{split}\end{equation}
Note that \begin{eqnarray*}
 &&\partial_U\eta(U|C)=\eta^\prime(U)-\eta^\prime(C),\\
 &&\partial_U F(U,C)=G'(U)-\eta'(C)A'(U)=(\eta^\prime(U)-\eta^\prime(C))A'(U).
 \end{eqnarray*}
 So, for any solution  $U$ of \eqref{scalar_cl} and any constant $C$, multiplying \eqref{scalar_cl} by $\eta^\prime(U)-\eta^\prime(C)$, we get 
\begin{equation}\label{relative_entropy_scalar_cl}
\partial_t\eta(U|C)+\partial_xF(U,C)=\eps(\eta^\prime(U)
-\eta^\prime(C))\partial_{xx}^2U
\end{equation} 
We have the following lemma.
\begin{lem}\label{baby_lemma}
The function $\He$, defined in (\ref{def_H}), satisfies the following on $(0,T)$
\begin{equation*}
\begin{split}
&\qquad\qquad\qquad \frac{d\He}{dt}(t)=\\
&{\int_{X(t)-\delta\eps}^{X(t)}
\frac{2}{\eps} 
\phi\Big(\frac{-x+X(t)}{\eps}\Big)
\phi^\prime\Big(\frac{-x+X(t)}{\eps}\Big)
\Big[\dot{X}(t)\eta(U(t,x)|C_L)-F(U(t,x),C_L)
\Big]dx}\\
& + {\eps \int_{-\infty}^{X(t)}
\Big[\phi\Big(\frac{-x+X(t)}{\eps}\Big)\Big]^2
\partial^2_{xx}U(t,x) (\eta^\prime(U(t,x))-\eta^\prime(C_L))dx}\\
\end{split}
\end{equation*}
\begin{equation*}
\begin{split}
&{-\int_{X(t)}^{X(t)+\delta\eps}
\frac{2}{\eps} 
\phi\Big(\frac{x-X(t)}{\eps}\Big)
\phi^\prime\Big(\frac{x-X(t)}{\eps}\Big)
\Big[\dot{X}(t)\eta(U(t,x)|C_R)-F(U(t,x),C_R)
\Big]dx}\\
& +{\eps \int_{X(t)}^\infty
\Big[\phi\Big(\frac{x-X(t)}{\eps}\Big)\Big]^2
\partial^2_{xx}U(t,x) (\eta^\prime(U(t,x))-\eta^\prime(C_R))dx}\\
&:={(L)_{\mbox{\it{Hyp}}}}+{(L)_{\mbox{\it{Dif}}}}+{(R)_{\mbox{\it{Hyp}}}}+{(R)_{\mbox{\it{Dif}}}}.
\end{split}
\end{equation*}
\end{lem}
\begin{proof}
First we split the term $\He(t)$ into the two parts: 
  \begin{equation*}\begin{split}
  &\He(t)=\int_{-\infty}^{\infty}
\Big[\phi\Big(\frac{|x-X(t)|}{\eps}\Big)\Big]^2\eta(U|S)dx\\
&=\int_{-\infty}^{\infty}\Big(
\Big[\phi\Big(\frac{-x+X(t)}{\eps}\Big)\Big]^2+\Big[\phi\Big(\frac{x-X(t)}{\eps}\Big)\Big]^2\Big)\eta(U|S)dx\\
&={\int_{-\infty}^{\infty}
\Big[\phi\Big(\frac{-x+X(t)}{\eps}\Big)\Big]^2\eta(U|C_L)dx}
+{\int^{\infty}_{-\infty}
\Big[\phi\Big(\frac{x-X(t)}{\eps}\Big)\Big]^2\eta(U|C_R)dx}\\
&:=H^L+H^R.
\end{split}\end{equation*}

To compute $\frac{d}{dt}(H^{L})$,  we put $C=C_L$ in \eqref{relative_entropy_scalar_cl}, multiply by
$\Big[\phi\Big(\frac{-x+X(t)}{\eps}\Big)\Big]^2$,
 and integrate in $x$. Then   we have
\begin{equation*}\begin{split}
\frac{d}{dt}(H^{L})
&=\int_{-\infty}^{\infty}
\partial_t
\Big(\Big[\phi\Big(\frac{-x+X(t)}{\eps}\Big)\Big]^2\Big) \eta(U|C_L)dx\\
&\quad +\int_{-\infty}^{\infty}
\partial_x\Big(\Big[\phi\Big(\frac{-x+X(t)}{\eps}\Big)\Big]^2\Big)
F(U,C_L)dx\\
&\quad +\eps\int_{-\infty}^{\infty}
\Big[\phi\Big(\frac{-x+X(t)}{\eps}\Big)\Big]^2
\partial^2_{xx}U (\eta^\prime(U)-\eta^\prime(C_L))dx
\end{split}\end{equation*} 
\begin{equation*}\begin{split}
&={\int_{X(t)-\delta\eps}^{X(t)}
(\frac{2}{\eps})
\cdot 
\phi\Big(\frac{-x+X(t)}{\eps}\Big)
\phi^\prime\Big(\frac{-x+X(t)}{\eps}\Big)
\Big[\dot{X}(t)\eta(U|C_L)-F(U,C_L)
\Big]dx}\\
&\quad +{\eps\cdot\int_{-\infty}^{X(t)}
\Big[\phi\Big(\frac{-x+X(t)}{\eps}\Big)\Big]^2
\partial^2_{xx}U (\eta^\prime(U)-\eta^\prime(C_L))dx}
=(L)_{\mbox{\it{Hyp}}}+(L)_{\mbox{\it{Dif}}}.
\end{split}\end{equation*}
We get the result for $\frac{d}{dt}(H^{R})={(R)_{\mbox{\it{Hyp}}}}+{(R)_{\mbox{\it{Dif}}}} $ 
in the same way.
\end{proof}

\section{Control of the hyperbolic terms}

In this section, we show that by choosing a special drift function $X(\cdot)$, the hyperbolic effects become
nonpositive. This will be used in section \ref{section_parabolic} to control 
the parabolic effects.\\

Following \cite{nick},  we define the normalized relative entropy flux 
$f(\cdot,\cdot)$ by 
$$f(x,y):=\frac{F(x,y)}{\eta(x|y)}.$$ We have the following properties.
 \begin{lem}\label{lem_basic_propperties} For any $L>0$, there exists  a constant
 $\Lambda>0$,  such that for any $x,y$ with $|x|,|y|\leq L$, we have
 \begin{equation}\begin{split}\label{properties_eta}.\\
&1/\Lambda \leq\eta^{\prime\prime}(x)\leq \Lambda,\\
&\frac{1}{2\Lambda}(x-y)^2 
\leq\eta(x|y)\leq \frac{1}{2}{\Lambda}(x-y)^2,\\
   &|F(x,y)|\leq \Lambda(x-y)^2,\\
 &0\leq (\partial_x f)(x,y)\leq \Lambda, \\
&1/\Lambda\leq(\partial_y f)(x,y).
 \end{split}\end{equation}
 \end{lem}
The proof of this lemma can be found in   \cite{nick}.
\vskip0.3cm
We now define the shift function $X$. It is the solution of the following O.D.E.
\begin{equation}
\begin{cases}\label{ode}
&\dot{X}(t)=f\Big(U(t,X(t)),\frac{C_L+C_R}{2}\Big)\\
&X(0)=0\end{cases}. \end{equation}
Note that for any $\eps>0$, $U\in C^1([0,T]\times\R)$ (since $U_0\in C^1(\R)$). The existence and uniqueness of $X$ comes from the Cauchy-Lipschitz theorem.
\vskip0.3cm
First, $X$ is Lipschitz, since we have from Lemma \ref{lem_basic_propperties}
\begin{equation}\begin{split}
|\dot{X}(t)|& 
\leq\frac{\Big|F\Big(U(t,X(t)),\frac{C_L+C_R}{2}\Big)\Big|}
{\eta\Big(U(t,X(t))\Big |\frac{C_L+C_R}{2}\Big)}
\leq
2\Lambda^2
\end{split}\end{equation} where we used the fact
$\|U(t)\|_{L^\infty}\leq\|U_0\|_{L^\infty}$ for $t>0$. It proves \eqref{estimate_curve_lishitz}.

\vskip0.3cm

Note that from the definition of $X$, if the velocity $U$ were constant in the layer (that is $U(t,x)\sim U(t,X(t))$ for $-\delta\eps\leq x-X(t)\leq \delta\eps$), then, from the
 last property of Lemma \ref{lem_basic_propperties}, we would have directly that 
$$
(L)_{\mbox{\it{Hyp}}}+(R)_{\mbox{\it{Hyp}}}\leq -\frac{C_L-C_R}{\Lambda} (\eta(U(t,X(t))|C_L)+\eta(U(t,X(t))|C_R)). 
$$
However, this is too much to hope, since the layer characterize 
the region where the function $U(t,\cdot)$ is expected 
to drop from about $C_L$ to about $C_R$. We still can show that the
hyperbolic terms are negative, provided that the behavior
of $U$ in the layer is not too much oscillatory 
(the values can drop, but not much bounce back). This last property of $U$ is proved in the following lemma which can be seen as a weak version of the Ole{\u\i}nik's principle
.
\begin{lem}\label{lem_decreasing_L_2_positive_derivative}
$\|(\partial_xU(t))_+\|_{L^2(\mathbb{R})}\leq\|(\frac{d}{dx}U_0)_+\|_{L^2(\mathbb{R})}$ for any $t>0$.
\end{lem}
\begin{proof}
We differentiate \eqref{scalar_cl} w.r.t. $x$, multiply $(\partial_xU)_+$ and
integrate in $x$ to get
\begin{equation*}\begin{split}
&0=\int(\partial_xU)_+ \Big[\partial_t\partial_xU+
A^{\prime\prime}(U) |\partial_xU|^2+A^{\prime}(U) \partial^2_{xx}
U-\eps \partial_{xxx}^3U\Big]dx\\
&=\int\Big[\frac{1}{2}\partial_t([(\partial_xU)_+]^2)
+A^{\prime\prime}(U) (\partial_xU)_+^3\\
&\quad\quad\quad\quad\quad\quad\quad
+A^{\prime}(U) \partial_{x}\Big(\frac{[(\partial_xU)_+]^2}{2}\Big)+\eps  
|\partial_x((\partial_xU)_+)|^2
\Big]dx.
 \end{split}\end{equation*} Then, we use the integration by parts to get
 \begin{equation*}\begin{split}
&=\int\Big[\frac{1}{2}\partial_t([(\partial_xU)_+]^2)
+\frac{1}{2}A^{\prime\prime}(U) (\partial_xU)_+^3
+\eps  
|\partial_x((\partial_xU)_+)|^2
\Big]dx\\
&\geq\frac{1}{2}\frac{d}{dt}\int[(\partial_xU)_+]^2
dx.\\
 \end{split}\end{equation*}

\end{proof}
\begin{rem}\label{L_P_lemma}
The result of the above lemma can be extended up to the case $L^p$ for any $1\leq p\leq\infty$. Indeed, for  any finite $p$, we just
multiply  
$\Big((\partial_xU)_+\Big)^{p-1}$ instead of $(\partial_xU)_+$ in the proof. Then the limit case
$p=\infty$ follows directly.
\end{rem}

We now prove the main proposition of this section.

\begin{prop}\label{prop_Hyp}
Let $(L)_{\mbox{\it{Hyp}}}$ and $(R)_{\mbox{\it{Hyp}}}$ be such as in Lemma \ref{baby_lemma}.
There exists a constant $\theta>0$ such that, for any $\eps, \delta$ satisfying
$$
\eps\delta\leq {\theta},
$$ 
we have 
\begin{eqnarray*}
&&\qquad\qquad (L)_{\mbox{\it{Hyp}}}+(R)_{\mbox{\it{Hyp}}}\\
&&\leq -\frac{\theta}{\eps}\int_{X(t)-\delta\eps}^{X(t)+\delta\eps}\phi\left(\frac{|x-X(t)|}{\eps}\right)\phi'\left(\frac{|x-X(t)|}{\eps}\right)(U(t,x)-S(t,x))^2\,dx.
\end{eqnarray*}

\end{prop}
\begin{proof}
We use the definition of $X(t)$ to get
\begin{equation*}\begin{split}
&(L)_{\mbox{\it{Hyp}}}
=\int_{X(t)-\delta\eps}^{X(t)}
(\frac{2}{\eps})
\cdot 
\phi\Big(\frac{-x+X(t)}{\eps}\Big)\cdot
\phi^\prime\Big(\frac{-x+X(t)}{\eps}\Big)
\cdot\eta(U|C_L)
\cdot h(t,x)dx
\end{split}\end{equation*} where $h(t,x):=\Big[f\Big(U(t,X(t)),\frac{C_L+C_R}{2}\Big)-f(U(t,x),C_L)
\Big]$.\\

In order to make the function $h(t,x)$ strictly negative over the
domain of the above integral, we use the condition
$(\frac{d}{dx}U_0)_+\in{L^2(\mathbb{R})}$. Indeed, we
observe that, for any $x\in[X(t)-\delta\eps,X(t)]$,
\begin{equation}\label{L_2_cond_used}\begin{split}
U(t,X(t))-U(t,x)&=\int_x^{X(t)}(\partial_xU)(t,y)dy
\leq\int_x^{X(t)}(\partial_xU)_+(t,y)dy\\
&\leq\|(\partial_xU(t))_+\|_{L^2(\mathbb{R})}\sqrt{|X(t)-x|}
\leq\|(\frac{d}{dx}U_0)_+\|_{L^2(\mathbb{R})}\sqrt{\delta\eps},
\end{split}\end{equation} where we used that $\|(\partial_xU(t))_+\|_{L^2}$ is not increasing (see Lemma \ref{lem_decreasing_L_2_positive_derivative}).

We can rewrite the function $h$ as
\begin{equation*}\begin{split}
h(t,x)&=f\Big(U(t,X(t)),\frac{C_L+C_R}{2}\Big)
-f\Big(U(t,x),\frac{C_L+C_R}{2}\Big)\\
&\quad\quad\quad\quad\quad
+f\Big(U(t,x),\frac{C_L+C_R}{2}\Big)
-f(U(t,x),C_L).\\
\end{split}\end{equation*}
 Since $f$ is increasing with respect to the first variable, we have
\begin{equation*}\begin{split}
&h(t,x)\leq f\Big(U(t,x)+\|(\frac{d}{dx}U_0)_+\|_{L^2(\mathbb{R})}\sqrt{\delta\eps},\frac{C_L+C_R}{2}\Big)
-f\Big(U(t,x),\frac{C_L+C_R}{2}\Big)\\
&\quad\quad\quad\quad\quad+f\Big(U(t,x),\frac{C_L+C_R}{2}\Big)
-f(U(t,x),C_L).\\
\end{split}\end{equation*} 
Then, thanks to Lemma \ref{lem_basic_propperties},
  we get
\begin{equation*}\begin{split}
&h(t,x)\leq \Lambda\|(\frac{d}{dx}U_0)_+\|_{L^2(\mathbb{R})}
\sqrt{\delta\eps}
-\frac{{C_L-C_R}}{2\Lambda}\leq -\theta<0\\
\end{split}\end{equation*}
for $\sqrt{\delta\eps}$ and
$\theta$ 
small enough.\\

\noindent Since $\phi(\cdot),\phi^\prime(\cdot)$ and $\eta(\cdot|\cdot)\geq0$,  we get
 \begin{equation*}\begin{split}
&(L)_{\mbox{\it{Hyp}}}\leq -\theta\int_{X(t)-\delta\eps}^{X(t) }
 \frac{2}{\eps}
\phi\Big(\frac{-x+X(t)}{\eps}\Big)
\phi^\prime\Big(\frac{-x+X(t)}{\eps}\Big)
\eta(U|C_L)  dx.\\
\end{split}\end{equation*} 
Then, from Lemma \ref{lem_basic_propperties}, we have (changing the constant $\theta$ if necessary)
 \begin{equation*}\begin{split}
&(L)_{\mbox{\it{Hyp}}}\leq-\theta \int_{X(t)-\delta\eps}^{X(t) }
( \frac{2}{\eps})
\phi\Big(\frac{-x+X(t)}{\eps}\Big)
\phi^\prime\Big(\frac{-x+X(t)}{\eps}\Big)
(U-C_L)^2
 dx.\\
\end{split}\end{equation*} 

In a similar way, we obtain the following estimate on $(II)_{Hyp}$.
 \begin{equation*}\begin{split}
&(R)_{\mbox{\it{Hyp}}}\leq-\theta \int_{X(t)}^{X(t)+\delta\eps }
( \frac{2}{\eps})
\phi\Big(\frac{x-X(t)}{\eps}\Big)
\phi^\prime\Big(\frac{x-X(t)}{\eps}\Big)
(U-C_R)^2
 dx.\\
\end{split}\end{equation*} 
Combining the two last inequalities gives the desired result.
\end{proof}

\section{Control of the parabolic terms}\label{section_parabolic}

For any $\delta\geq 1/\theta$, we now fix the function $\phi$  in the following explicit way. 
\begin{equation}\label{pd}
\phi(x)=
\begin{cases}
& \theta e^{1-\theta\delta}x,\qquad \mathrm{for} \ \ x\in [0,1/\theta),\\
&e^{\theta(x-\delta)},\qquad \mathrm{for} \ \ x\in [1/\theta, \delta].
\end{cases}
\end{equation}
We will 
use the straightforward computation:
\begin{equation}\label{computation_pd}
\int_0^\delta (\phi'(x))^2 \chi_{\{\phi'>\theta\phi\}}\,dx=C_\theta \cdot e^{-2\theta \delta}.
\end{equation} 

This section is dedicated to the proof of the following proposition.
\begin{prop}\label{prop_diffusion}
There exists constants $\theta, C>0$ such that for any $\eps,\delta$ verifying
$$
\frac{1}{\theta}\leq \delta\qquad\mbox{and}\qquad \eps\delta\leq {\theta},
$$
we have 
$$
\frac{d\He(t)}{dt}\leq Ce^{-\theta\delta}.
$$
\end{prop}
\begin{proof}
First, we estimate the term  $(L)_{\mbox{\it{Dif}}}$. Integrating  by parts,  we obtain
\begin{equation*}\begin{split}
(L)_{\mbox{\it{Dif}}}&=\int_{-\infty}^{X(t)}
 2\phi\Big(\frac{-x+X(t)}{\eps}\Big) 
\phi^\prime\Big(\frac{-x+X(t)}{\eps}\Big) 
\partial_{x}U  (\eta^\prime(U)-\eta^\prime(C_L))
dx\\
&\quad -2\eps \int_{-\infty}^{X(t)}
\Big[\phi\Big(\frac{-x+X(t)}{\eps}\Big)\Big]^2 
\eta^{\prime\prime}(U) 
|\partial_{x}U|^2dx\\
\end{split}\end{equation*}
 Then, using H\"{o}lder's inequality and  Lemma \ref{lem_basic_propperties}, we get
 \begin{equation*}\begin{split}
&(L)_{\mbox{\it{Dif}}}\leq
\frac{2\eps}{\Lambda} \int_{-\infty}^{X(t)}
\Big[\phi\Big(\frac{-x+X(t)}{\eps}\Big)\Big]^2 
|\partial_{x}U|^2dx\\
&\quad+\frac{\Lambda}{8\eps}\int_{\infty}^{X(t)}
\Big[2 \phi^\prime\Big(\frac{-x+X(t)}{\eps}\Big)(\eta^\prime(U)-\eta^\prime(C_L))\Big]^2dx\\
&\quad-\frac{2\eps}{\Lambda} \int_{-\infty}^{X(t)}
\Big[\phi\Big(\frac{-x+X(t)}{\eps}\Big)\Big]^2  
|\partial_{x}U|^2dx\\
&\leq\frac{C}{\eps}\int_{X(t)-\delta\eps}^{X(t)}
\Big[ \phi^\prime\Big(\frac{-x+X(t)}{\eps}\Big)\Big]^2
|U-C_L|^2 \,dx.
\end{split}\end{equation*}
In the same way, we obtain the following estimate for $(R)_{\mbox{\it{Dif}}}$.
$$
(R)_{\mbox{\it{Dif}}}\leq \frac{C}{\eps}\int_{X(t)}^{X(t)+\delta\eps}
\Big[ \phi^\prime\Big(\frac{x-X(t)}{\eps}\Big)\Big]^2
|U-C_R|^2 \,dx.
$$
Combining the two last inequalities, we find
\begin{equation}\label{estimate_diffusion}
(L)_{\mbox{\it{Dif}}}+(R)_{\mbox{\it{Dif}}}\leq \frac{C}{\eps}\int_{X(t)-\delta\eps}^{X(t)+\delta\eps}
\Big[ \phi^\prime\Big(\frac{|x-X(t)|}{\eps}\Big)\Big]^2
|U(t,x)-S(t,x)|^2 \,dx.
\end{equation}

\vskip0.3cm
\noindent Using Lemma \ref{baby_lemma}, Proposition \ref{prop_Hyp}, and (\ref{estimate_diffusion}), we find
\begin{equation}\begin{split}\label{eq_keep_charac}
&\frac{d\He(t)}{dt}\leq
 \frac{1}{\eps}\int_{X(t)-\delta\eps}^{X(t)+\delta\eps}
\Big[\phi^\prime (C\phi'-\theta\phi)\Big]\left(\frac{|x-X(t)|}{\eps}\right)|U(t,x)-S(t,x)|^2\,dx.
 \end{split}\end{equation} 
 Using that $U-S$ is a bounded function, and doing the change of 
variables $z=(x-X(t))/\eps$, we find:
\begin{equation*}\begin{split}
 \frac{d\He(t)}{dt} &
 \leq  \frac{C}{\eps}\int_{X(t)-\delta\eps}^{X(t)+\delta\eps}
\Big[(\phi^\prime)^2 \chi_{\{C\phi'-\theta\phi>0\}}\Big] \left(\frac{|x-X(t)|}{\eps}\right)|U(t,x)-S(t,x)|^2\,dx\\
&\leq \frac{C\|U(t)-S(t)\|^2_{L^\infty}}{\eps}\int_{X(t)-\delta\eps}^{X(t)+\delta\eps}
\Big[(\phi^\prime)^2 \chi_{\{C\phi'-\theta\phi>0\}}\Big] \left(\frac{|x-X(t)|}{\eps}\right) \,dx\\
&\leq C\int_{0}^{\delta}
(\phi^\prime)^2(z) \chi_{\{C\phi'-\theta\phi>0\}}(z)\,dz.
 \end{split}\end{equation*} 
 Changing the constant $\theta$ if needed, and using (\ref{computation_pd}), gives the desired result.
 \end{proof}
 
 \section{Proof of Theorem \ref{main_thm}}
 This section is devoted to the proof of Theorem \ref{main_thm}. 
  Integrating the estimate of Proposition \ref{prop_diffusion} between $0$
  and $t\in(0,T)$ gives the result of (I). 
  Indeed,  for any $\eps,\delta$ with 
$
\frac{1}{\theta}\leq \delta\mbox{ and } \eps\delta\leq {\theta},
$
 where 
  $\theta$ is the constant from Proposition \ref{prop_diffusion},
  we have
  \begin{equation*}\begin{split}
  \int_{\{|x-X(t)|\geq \delta\eps\}}\eta(U(t,x)|S(t,x))\,dx
  &\leq \He(t)
  \leq \He(0)  +\int_0^t \frac{d}{dt}\He(s)\,ds\\
  &\leq
\int_\R\eta(U_0|S_0)\,dx
 +CT e^{-\theta\delta}
 \end{split}\end{equation*}   
  By taking $\eps_0:=\theta^2$, 
     we have  for any $\eps\leq\alpha\leq\eps_0$, 
  \begin{equation*}\begin{split}
  \int_{\{|x-X(t)|\geq \alpha/\theta\}}\eta(U(t,x)|S(t,x))\,dx
   &\leq
\int_\R\eta(U_0|S_0)\,dx
 +CT e^{-\alpha/\eps}.
 \end{split}\end{equation*} It proves our main estimate \eqref{main_estiamte}
 by taking $\C$ large enough.\\

Observe that
\begin{equation*}\begin{split}
 \int_{\R} \eta(U|S)\,dx&=
 \int_{\{|x-X(t)|\geq \C\alpha\}}\eta(U|S)\,dx
 +\int_{\{|x-X(t)|< \C\alpha\}}\eta(U|S)\,dx\\ 
\end{split}\end{equation*}
and the second term is bounded by $C\C\alpha$.
  Thus, by taking $\alpha=\eps\log(1/\eps)$, we obtain for any $t\in(0,T)$,
 \begin{equation}\label{estimate_temp}
  \int_{\R} \eta(U|S)\,dx\leq \int_{\R} \eta(U_0|S_0)\,dx+\C\eps\log(1/\eps)
 \end{equation} for any $\eps\leq\eps_0$ 
(changing $\eps_0$ and $\C$ if needed).
  
 \vskip0.3cm
  It only remains to prove \eqref{estimate_curve_w.r.t._L^2_differ}. 
We define first
$\psi(x):=\begin{cases}
&0 \mbox{ if } |x|>2,\\
&1 \mbox{ if } |x|\leq1\\
&2-|x| \mbox{ if } 1<|x|\leq2
\end{cases}$. Let $s\in(0,t)$ and $R>0$. We multiply $\Psi_R(s,x):=\psi(\frac{x-X(s)}{R})$
to the  equation  \eqref{scalar_cl} and integrate in $x$ to get
\begin{equation*}\begin{split}
0=&-\frac{d}{ds}\int \Psi_R\cdot U dx+\int\partial_x(\Psi_R) A(U) dx
+\int\partial_t(\Psi_R) U dx+\eps\int\Psi_R
\cdot\partial^2_{xx}Udx\\
&=-
\underbrace{\frac{d}{ds}\int \psi(\frac{x-X(s)}{R})\cdot U(s,x) dx}_{(I)}\\
&\quad\quad+\underbrace{\frac{1}{R}\int\psi^\prime(\frac{x-X(s)}{R})\cdot \Big(
A(U(s,x))-\dot{X}(s)U(s,x) \Big)dx}_{(II)}\\
&\quad\quad-\underbrace{\eps\frac{1}{R}\int\psi^\prime(\frac{x-X(s)}{R})
\cdot\partial_{x}U(s,x)dx}_{(III)}.
 \end{split}\end{equation*}
By using the above observation,
 we have
\begin{equation*}\begin{split}
&(\sigma-\dot{X}(s))=\frac{1}{C_L-C_R}\Big(A(C_L)-A(C_R)-(C_L-C_R)\dot{X}(s)\Big)\\
&=\frac{1}{C_L-C_R}\Big(A(C_L)-A(C_R)-(C_L-C_R)\dot{X}(s)
-(II)+(I)+(III)\Big).\\
 \end{split}\end{equation*} Then we integrate the above equation in time on $[0,t]$ 
to get:
\begin{equation}\begin{split}\label{estimate_sum}
|\sigma t-{X}(t)|&\leq
C\Big(t\cdot\max_{s\in(0,t)}\underbrace{\Big|A(C_L)-A(C_R)-(C_L-C_R)\dot{X}(s)
-(II)\Big|}_{(II^\prime)}\\
&\quad+\Big|\int_0^t(I)ds\Big|+t\cdot\max_{s\in(0,t)}\Big|(III)\Big|\Big).\\
 \end{split}\end{equation}

We observe 
\begin{equation*}\begin{split}
{(II^\prime)}\leq&\underbrace{\Big|A(C_L)-A(C_R)
-{\frac{1}{R}\int\psi^\prime(\frac{x-X(s)}{R})\cdot 
A(U)dx}\Big|}_{(II^\prime_1)}\\
&+\underbrace{\Big|-(C_L-C_R)\dot{X}(s)+\frac{1}{R}\int\psi^\prime(\frac{x-X(s)}{R})\cdot \Big(
\dot{X}(s)U(s,x) \Big)dx\Big|}_{(II^\prime_2)}.\\
 \end{split}\end{equation*}
 For  $(II^\prime_1)$, we compute
\begin{equation*}\begin{split}
(II^\prime_1)
&=\Big|A(C_L)
-{\frac{1}{R}\int_{-2R+X(s)}^{-R+X(s)}
A(U)dx}
-A(C_R)
+{\frac{1}{R}\int_{R+X(s)}^{2R+X(s)}
A(U)dx}\Big|\\
&\leq
{\frac{1}{R}\Big[\int_{-2R+X(s)}^{-R+X(s)}
|A(C_L)-A(U)|dx}
+{\int_{R+X(s)}^{2R+X(s)}
|A(U)-A(C_R)|dx}\Big].
\end{split}\end{equation*} We use  $|A(y)-A(z)|\leq
C
|y-z|$ for $|y|,|z|\leq {M_1}$ to get
\begin{equation*}\begin{split}
\leq
{\frac{
C
}{R}\int_{-2R+X(s)}^{2R+X(s)}
|U-S|dx}.
\end{split}\end{equation*} We use   H\"{o}lder's inequality
and Lemma
 \ref{lem_basic_propperties} to get
\begin{equation*}\begin{split}
(II^\prime_1)^2&
\leq\frac{C}{{R}}\cdot
{\int_{\R}\eta(U(s)|S(s))dx}.
 \end{split}\end{equation*} 
Likewise, for the second term $(II^\prime_2)$, we have
\begin{equation*}\begin{split}
(II^\prime_2)&=|\dot{X}(s)|\cdot\Big|-(C_L-C_R)
+\frac{1}{R}\int\psi^\prime(\frac{x-X(s)}{R})\cdot 
U(s,x) dx\Big|\\
&\leq
{\frac{
C
}{R}\int_{-2R+X(s)}^{2R+X(s)}
|U-S|dx}
\leq\frac{C}{\sqrt{R}}\cdot
\|U(s)-S(s)\|_{L^2(\mathbb{R})}
 \end{split}\end{equation*} where we used $|\dot{X}(s)|\leq
C$.
Thus we get 
\begin{equation}\begin{split}\label{estimate_(II)}
(II^\prime)^2\leq\frac{C}{{R}}\cdot
{\int_{\R}\eta(U(s)|S(s))dx}.
 \end{split}\end{equation}
On the other hand, we compute
\begin{equation*}\begin{split}
&\Big|\int_0^t(I)ds\Big|
=\Big|\int \psi(\frac{x-X(t)}{R})\cdot U(t,x) dx-
\int \psi(\frac{x}{R})\cdot U_0(x) dx\Big|\\
&=\Big|\int \psi(\frac{x-X(t)}{R})\cdot \Big(U(t,x)-S(t,x)\Big) dx+
\int \psi(\frac{x-X(t)}{R})\cdot S(t,x) dx\\
&-\int \psi(\frac{x}{R})\cdot S_0(x) dx
-\int \psi(\frac{x}{R})\cdot \Big(U_0(x)-S_0(x) \Big)dx\Big|.
 \end{split}\end{equation*} Note that  
$\int \psi(\frac{x-X(t)}{R})\cdot S(t,x) dx=\int \psi(\frac{x}{R})\cdot S_0(x) dx.$ Thus, we have
\begin{equation*}\begin{split}
&\leq\Big|\int \psi(\frac{x-X(t)}{R})\cdot \Big(U(t,x)-S(t,x)\Big) dx\Big|
+\Big|\int \psi(\frac{x}{R})\cdot \Big(U_0(x)-S_0(x) \Big)dx\Big|.
 \end{split}\end{equation*} We use  H\"{o}lder's inequality and Lemma
 \ref{lem_basic_propperties}
 to get
\begin{equation}\begin{split}\label{estimate_(I)}
\Big|\int_0^t(I)ds\Big|^2&\leq C{R}\Big({\int_{\R}\eta(U(t)|S(t))dx}+
{\int_{\R}\eta(U_0|S_0)dx}\Big).
 \end{split}\end{equation}  
Also, we have
\begin{equation}\begin{split}\label{estimate_(III)}
&\Big|(III)\Big|=
\frac{\eps}{R}\Big|\int\psi^\prime(\frac{x-X(s)}{R})
\cdot\partial_{x}U(s,x)dx\Big|\\
&=
\frac{\eps}{R}\Big|\int_{-2R+X(s)}^{-R+X(s)}
\partial_{x}U(s,x)dx-
\int_{R+X(s)}^{2R+X(s)}
\partial_{x}U(s,x)dx\Big|\\
&\leq \frac{\eps}{R}\cdot 4\cdot\|U(s)\|_{L^\infty}
\leq \frac{C\cdot\eps}{R}.
 \end{split}\end{equation}

Finally, by using \eqref{estimate_temp}, we combine 
\eqref{estimate_(II)}, \eqref{estimate_(I)} and  \eqref{estimate_(III)} with 
\eqref{estimate_sum}
to get, for any  $t\in(0,T)$,
\begin{equation*}\begin{split}
|\sigma t-{X}(t)|^2&\leq C\Big(
\frac{ t^2}{{R}}+
{R}
\Big)\cdot\Big(
{\int_{\R}\eta(U_0|S_0)dx}+{
\eps\log(\frac{1}{\eps})}\Big)
+\frac{C\cdot\eps^2\cdot t^2}{R^2}.\\
 \end{split}\end{equation*} Since the above estimate holds for any $0<R<\infty$, 
 the estimate
 \eqref{estimate_curve_w.r.t._L^2_differ} follows once we take  $R:=t^{2/3}$
(changing $\C$ if needed).

\begin{ack}
The first author has been partially supported by the National Science
Foundation (NSF) grant DMS-1159133.
The second author was partially funded by the NSF.
 We would like to thank the anonymous referees for many useful suggestions.

\end{ack}

\bibliographystyle{plain}
 \def\cprime{$'$}

\end{document}